\documentclass{article} %
\usepackage{nips13submit_e,times}

\usepackage{graphicx, amsmath, amssymb,amsthm}
\usepackage[breaklinks]{hyperref} %
\usepackage{url}
\usepackage{bm,color}
\usepackage{mathdots}

\usepackage{wrapfig}
\usepackage{multicol}
\usepackage[boxruled,commentsnumbered]{algorithm2e}
\newenvironment{algo}{%
  \algorithm
}{%
  \endalgorithm
}

\definecolor{darkgreen}{rgb}{0,.4,.2}
\definecolor{darkblue}{rgb}{.1,.2,.6}
\definecolor{brightblue}{rgb}{0,0.6,0.8}
\hypersetup{
  colorlinks=true,
  linkcolor=darkblue,
  citecolor=darkgreen,
  filecolor=darkblue,
  urlcolor=darkblue
}

\newtheorem*{rep@theorem}{\rep@title}
\newcommand{\newreptheorem}[2]{%
\newenvironment{rep#1}[1]{%
 \def\rep@title{#2 \ref{##1}}%
 \begin{rep@theorem}}%
 {\end{rep@theorem}}}

\newreptheorem{lemma}{Lemma'}
\newreptheorem{proposition}{Proposition'}
\newreptheorem{theorem}{Theorem'}
\newreptheorem{remark}{Remark'}

\newtheorem{definition}{Definition}
\newtheorem{theorem}[definition]{Theorem}

\newtheorem{lemma}[definition]{Lemma}

\newtheorem{remark}[definition]{Remark}

\DeclareMathOperator*{\conv}{conv}
\DeclareMathOperator{\diam}{diam}

\DeclareMathOperator*{\argmin}{\arg\min}
\DeclareMathOperator*{\argmax}{\arg\max}

\providecommand{\norm}[1]{\left\lVert#1\right\rVert}
\providecommand{\dualnorm}[1]{\norm{#1}_*}

\newcommand{\R}{\mathbb{R}}
\newcommand{\X}{\mathcal{X}}

\newcommand{\domain}{\mathcal{D}}
\newcommand{\stepsize}{\gamma}
\newcommand{\stepmax}{\stepsize_{\textrm{max}}} %
\newcommand{\stepbound}{\stepsize^\textrm{B}} %
\newcommand{\FW}{{\hspace{0.05em}\textsf{FW}}}
\newcommand{\away}{{\hspace{0.06em}\textsf{A}}}
\newcommand{\Cf}{C_{\hspace{-0.08em}f}}
\newcommand{\CfA}{C_{\hspace{-0.08em}f}^-}
\newcommand{\CfMFW}{C_{\hspace{-0.08em}f}^\away}
\newcommand{\strongConvFW}{\mu_{\hspace{-0.08em}f}^\FW}
\newcommand{\strongConvMFW}{\mu_{\hspace{-0.08em}f}^\away}
\newcommand{\distToBoundary}{{\delta_{\x^*\!,\domain}}}
\newcommand{\dirW}{\mathop{dirW}}
\newcommand{\x}{\bm{x}}
\newcommand{\y}{\bm{y}}

\newcommand{\s}{\bm{s}}
\newcommand{\dd}{\bm{d}}

\newcommand{\vv}{\bm{v}} %

\newcommand{\mapprox}{\nu} %

\newcommand{\Vertices}{\mathcal{V}}
\newcommand{\Coreset}{\mathcal{S}}
\renewcommand{\S}{\mathcal{S}}
\renewcommand{\aa}{\bm{\alpha}}
\renewcommand{\r}{\bm{r}}
\newcommand{\PdirW}{\mathop{PdirW}}
\newcommand{\innerProd}[2]{\left\langle #1 , #2 \right\rangle}
\newcommand{\C}{\mathcal{C}}
\newcommand{\proj}{\bm{P}}

\newcommand{\Kface}{\mathcal{K}}

\newcommand{\id}{\mathbf{I}} %
\newcommand{\0}{\mathbf{0}} %
\newcommand{\unit}{\mathbf{e}} %

\newcommand{\ignore}[1]{}%

\newcommand{\todo}[1]{\marginpar[\hspace*{4.5em}\textbf{TODO}\hspace*{-4.5em}]{\textbf{TODO}}\textbf{TODO:} #1}
\newcommand{\note}[1]{{\textbf{\color{red}#1}}}
\newcommand{\remove}[1]{} %

\title{An Affine Invariant Linear Convergence Analysis\\ for Frank-Wolfe Algorithms}

\author{
Simon Lacoste-Julien \\
INRIA - SIERRA project-team\\
{\'E}cole Normale Sup{\'e}rieure, Paris, France \\
\And
Martin Jaggi \\
Simons Institute for the Theory of Computing \\
UC Berkeley, USA \\
}

\nipsfinalcopy %

\begin{document}

\maketitle
\vspace{-2mm}

\begin{abstract}\vspace{-2mm}
We study the linear convergence of variants of the Frank-Wolfe algorithms for some classes of strongly convex problems, using only affine-invariant quantities. 
As in~\cite{Guelat:1986fq}, we show the linear convergence of the standard Frank-Wolfe algorithm when the solution is in the interior of the domain, but with affine invariant constants. We also show the linear convergence of the away-steps variant of the Frank-Wolfe algorithm, but with constants which only depend on the geometry of the domain, and not any property of the location of the optimal solution.
Running these algorithms does not require knowing any problem specific parameters.
\end{abstract}

The Frank-Wolfe algorithm \cite{Frank:1956vp} (also known as \emph{conditional gradient}) is one of the earliest existing methods for constrained convex optimization, and has seen an impressive revival recently due to its nice properties compared to projected or proximal gradient methods, in particular for sparse optimization and machine learning applications.

On the other hand, the classical projected gradient and proximal methods have been known to exhibit a very nice adaptive acceleration property, 
namely that the the convergence rate becomes linear for strongly convex objective, i.e. that  the optimization error of the same algorithm after $k$ iterations will decrease geometrically with $O(\rho^{-k})$ instead of the usual $O(1/k)$ for general convex objective functions.
It has become an active research topic recently whether such an acceleration is also possible for Frank-Wolfe type methods.

\vspace{-2mm}
\paragraph{Contributions.}
We show that the Frank-Wolfe algorithm with away-steps converges linearly (i.e. with a geometric rate) for any strongly convex objective function optimized over a polytope domain, with a constant bounded away from zero that only depends on the geometry of the polytope. Our convergence analysis is affine invariant (both the algorithm and the convergence rate are unaffected by an affine transformation of the variables). 
Also, our analysis does not depend on the location of the true optimum with respect to the domain, which was a disadvantage of earlier existing results such as \cite{Wolfe:1970wy,Guelat:1986fq,Beck:2004jm}, and the later results \cite{Ahipasaoglu:2008il,Kumar:2010ku,Allende:2013vw} that need Robinson's condition \cite{Robinson:1982ii}. Our analysis yields a weaker sufficient condition than Robinson's condition; in particular we can have linear convergence even in some cases when the function has more than one global minima and is not globally strongly convex.
As a second contribution, we provide an affine invariant version of the analysis of~\cite{Guelat:1986fq} showing that the classical (unmodified) Frank-Wolfe algorithm converges linearly on strongly convex functions when the optimum lies in the interior.

\vspace{-2mm}
\paragraph{Related Work.}
The away-steps variant of the Frank-Wolfe algorithm, that can also remove weight from ``bad'' ones of the currently active atoms, was proposed in \cite{Wolfe:1970wy}, and later also analyzed in \cite{Guelat:1986fq}. The precise algorithm is stated below in Algorithm \ref{alg:MFW}.
An alternative away-step algorithm (with a sublinear convergence rate) has been considered by \cite{Clarkson:2010hv}, namely performing an away step whenever the number of atoms of non-zero weight has exceeded a fixed target size. The disadvantage of this method is that it requires knowledge of the curvature constant, which is not realistic in many practical applications. %
For the classical Frank-Wolfe algorithm, the early work of \cite[Theorem 6.1]{Levitin:1966gf} has shown a linear convergence rate under the strong requirement that the objective is strongly convex, and furthermore the domain is strongly convex as a set. %
\cite{Beck:2004jm} has shown a linear rate for the special case of quadratic objectives when the optimum is in the strict interior of the domain, 
but their result was already subsumed by~\cite{Guelat:1986fq}.
More recently \cite{Ahipasaoglu:2008il,Kumar:2010ku,Allende:2013vw} have obtained linear convergence results in the case that the optimum solution satisfies Robinson's condition \cite{Robinson:1982ii}.
In a different recent line of work, \cite{Garber:2013vl,Garber:2013dl} has studied an algorithm variation\footnote{This can be interpreted as a concrete instantiation of the stronger oracle proposed in~\cite{Lan:2013um}.} that moves mass from the worst vertices to the ``towards'' vertex until a specific condition is satisfied, yielding a linear convergence rate. %
Their algorithm requires the knowledge of several constants though, and moreover is not adaptive to the best-case scenario, unlike the Frank-Wolfe algorithm with away steps and line-search. 
None of these previous works was shown to be affine invariant, and most require additional knowledge about problem specific parameters.%

\vspace{-2mm}
\section{Frank-Wolfe Algorithms, and Away-Steps}\vspace{-2mm}

\begin{wrapfigure}{r}{9.3cm}\vspace{-1.5em} %
\begin{algo}[H]
  \caption{Frank-Wolfe Algorithm with Away Steps}
  \label{alg:MFW}
\SetAlgoNoLine\DontPrintSemicolon
  Let $\x^{(0)} \in \Vertices$, and $\Coreset^{(0)} := \{\x^{(0)}\}$ \vspace{-1mm}\;
   \hspace{3cm}\emph{\small(so that $\alpha^{(0)}_{\vv} = 1$ for $\vv=\x^{(0)}$ and $0$ otherwise)} \vspace{-1mm}\;
  \For{$k=0\dots K$}{
  Let $\s_k \in \displaystyle\argmin_{\vv \in \Vertices} \textstyle\left\langle \nabla f(\x^{(k)}), \vv \right\rangle$ and $\dd_k^\FW := \s_k - \x^{(k)}$ \vspace{-2.5mm}\;
    \hspace{6.3cm}\emph{\small(the FW direction)} \;
  Let $\vv_k \in \displaystyle\argmax_{\vv \in \Coreset^{(k)} } \textstyle\left\langle \nabla f(\x^{(k)}), \vv \right\rangle$ and $\dd_k^\away := \x^{(k)} - \vv_k$ \vspace{-3mm}\;
    \hspace{6.1cm}\emph{\small(the away direction)} \;
  \eIf{$\left\langle \nabla f(\x^{(k)}), \dd_k^\FW\right\rangle  \leq \left\langle \nabla f(\x^{(k)}), \dd_k^\away\right\rangle$ }{
  	$\dd_k :=  \dd_k^\FW$, and $\stepmax := 1$  
	     \hspace{5mm}\emph{\small(choose the FW direction)} \;
  	}{
	$\dd_k :=  \dd_k^\away$, and $\stepmax := \alpha_{\vv_k} / (1- \alpha_{\vv_k})$   \;
	     \hspace{0.5cm}\emph{\small(choose the away direction, and maximum feasible step-size)} \vspace{-2mm}\;
  }	
  Line search: $\stepsize_k \in \displaystyle\argmin_{\stepsize \in [0,\stepmax]} \textstyle f\left(\x^{(k)} + \stepsize \dd_k\right)$ \;
  Update $\x^{(k+1)} := \x^{(k)} + \stepsize_k \dd_k$ \;
       \hspace{2cm}\emph{\small(and accordingly for the weights $\aa^{(k+1)}$, see text)} \;
  Update $\Coreset^{(k+1)} := \{\vv \: s.t. \: \alpha^{(k+1)}_{\vv} > 0\}$\;
  }\vspace{-1mm}%
\end{algo}
\end{wrapfigure}

We consider general constrained convex optimization problems of the form\vspace{-1mm}
\begin{equation*}\label{eq:optGenConvex}
   \min_{\x \in \domain} \, f(\x) \ .\vspace{-1mm}
\end{equation*}
We assume~$f$ is convex and differentiable, and that the domain~$\domain$ %
is a bounded convex subset of a vector space.
The Frank-Wolfe method~\cite{Frank:1956vp}, also known as \emph{conditional gradient}~\cite{Levitin:1966gf} works as follows:
At a current~$\x^{(k)}$, the algorithm considers the linearization of the objective function, and moves slightly towards a minimizer of this linear function (taken over the same domain).
In terms of convergence, it is known that the iterates of Frank-Wolfe satisfy $f(\x^{(k)}) - f(\x^*) \le O\big(1/k\big)$, for $\x^*$ being an optimal solution~\cite{Frank:1956vp,Dunn:1978di,Jaggi:2013wg}. %
One of the main reasons for the recent increased popularity of Frank-Wolfe-type algorithms is the sparsity of the iterates, i.e. that the iterate is always represented as a sparse convex combination of at most $k$ vertices $\Coreset^{(k)}\subseteq \Vertices$ of the domain~$\domain$, which we write as $\x^{(k)}= \sum_{\vv \in \Coreset^{(k)}} \alpha^{(k)}_{\vv} \vv$.
Here $\Vertices$ is defined to be the set of vertices (extreme points) of $\domain$, so that $\domain = \conv(\Vertices)$.
We assume that the \emph{linear oracle} defining $\s_k$ always returns a point from $\Vertices$ as a minimizer.

\vspace{-4mm}
\paragraph{Away-Steps.}
The away-steps variant of Frank-Wolfe, as stated in Algorithm \ref{alg:MFW}, was proposed in \cite{Wolfe:1970wy}, with the idea to also remove weight from ``bad'' ones of the currently active atoms. %
Note that the classical Frank-Wolfe algorithm is obtained by only using the FW direction in Algorithm~\ref{alg:MFW}.
If $\stepsize_k = \stepmax$, then we call this step a \emph{drop step}, as it fully removes the vertex $\vv_k$ from the currently active set of atoms  $\Coreset^{(k)}$.
The updates of the algorithm are of the following form:
For a FW step, we have $\Coreset^{(k+1)} = \{\s_k\}$ if $\stepsize_k = 1$; otherwise $\Coreset^{(k+1)} = \Coreset^{(k)} \cup \{\s_k\}$. Also, we have $\alpha^{(k+1)}_{\s_k} := (1-\stepsize_k) \alpha^{(k)}_{\s_k} + \stepsize_k$ and $\alpha^{(k+1)}_{\vv} := (1-\stepsize_k) \alpha^{(k)}_{\vv}$ for $\vv \in \Coreset^{(k)} \setminus  \{\s_k\}$. 
For an away step, we have $\Coreset^{(k+1)} = \Coreset^{(k)} \setminus \{\vv_k\}$ if  $\stepsize_k = \stepmax$ (a \emph{drop step}); 
otherwise $\Coreset^{(k+1)} = \Coreset^{(k)}$.  Also, we have $\alpha^{(k+1)}_{\vv_k} := (1+\stepsize_k) \alpha^{(k)}_{\vv_k} - \stepsize_k$ and $\alpha^{(k+1)}_{\vv} := (1+\stepsize_k) \alpha^{(k)}_{\vv}$ for $\vv \in \Coreset^{(k)} \setminus  \{\vv_k\}$.

\vspace{-2mm}
\section{Affine Invariant Measures of Smoothness and Strong Convexity}\vspace{-2mm}
\paragraph{Affine Invariance.}
An optimization method is called \emph{affine invariant} if it is invariant under affine transformations of the input problem: If one chooses any re-parameterization of the domain~$\domain$, by a \emph{surjective} linear or affine map $M:\hat\domain\rightarrow\domain$, then the ``old'' and ``new'' optimization problems $\min_{\x\in\domain}f(\x)$ and $\min_{\hat\x\in\hat\domain}\hat f(\hat\x)$ for $\hat f(\hat\x):=f(M\hat\x)$ look completely the same to the algorithm.
More precisely, every ``new'' iterate must remain exactly the transform of the corresponding old iterate; an affine invariant analysis should thus yield the convergence rate and constants unchanged by the transformation. It is well known that Newton's method is affine invariant under invertible $M$, and the Frank-Wolfe algorithm is affine invariant in the even stronger sense under arbitrary $M$ \cite{Jaggi:2013wg}. (This is directly implied if the algorithm and all constants appearing in the analysis only depend on inner products with the gradient, which are preserved since $\nabla \hat f = M^T\nabla f$.)
\vspace{-2mm}
\paragraph{Affine Invariant Measures of Smoothness.}
The affine invariant convergence analysis of the standard Frank-Wolfe algorithm by \cite{Jaggi:2013wg} crucially relies on the following measure of non-linearity of the objective function $f$ over the domain $\domain$. The \emph{curvature constant} $C_{f}$ of a convex and differentiable function $f:\R^n\rightarrow\R$, with respect to a compact domain $\domain$ is defined as\vspace{-1mm}
\begin{equation}\label{eq:Cf}
  \Cf := \sup_{\substack{\x,\s\in \domain,  ~\stepsize\in[0,1],\\
                      \y = \x+\stepsize(\s-\x)}} \textstyle
           \frac{2}{\stepsize^2}\big( f(\y)-f(\x)-\innerProd{\nabla f(\x)}{\y-\x}\big) \ .\vspace{-4mm}
\end{equation}

The assumption of bounded curvature $\Cf$ closely corresponds to a Lipschitz assumption on the gradient of~$f$. %
More precisely, if $\nabla f$ is $L$-Lipschitz continuous on $\domain$ with respect to some arbitrary chosen norm $\norm{.}$, then\vspace{-2mm}
\begin{equation} \label{eq:CfBound}
\Cf \le \diam_{\norm{.}}(\domain)^2 L \ ,\vspace{-0.5mm}
\end{equation}
where $\diam_{\norm{.}}(.)$ denotes the $\norm{.}$-diameter, see \cite[Lemma 7]{Jaggi:2013wg}.
While the early papers \cite{Frank:1956vp,Dunn:1979da} on the Frank-Wolfe algorithm relied on such Lipschitz constants with respect to a norm, 
the curvature constant $\Cf$ here is affine invariant, does not depend on any norm, and gives tighter convergence rates. 
$\Cf$ combines the complexity of $\domain$ and the curvature of $f$ into a single quantity.

\vspace{-2mm}
\paragraph{An Affine Invariant Notion of Strong Convexity.}
Inspired by the affine invariant curvature measure, one can also define a related affine invariant measure of strong convexity, when combined with the assumption of the optimum $\x^*$ being in the strict interior of $\domain$:\vspace{-2mm}
\begin{equation}\label{eq:mufFW}
  \strongConvFW := \inf_{\substack{\x\in \domain\setminus\{\x^*\}, ~\stepsize\in(0,1],\\
                      \overline\s = \overline\s(\x, \x^*,\domain),\\%
                      \y = \x+\stepsize(\overline\s-\x)}} \textstyle
           \frac{2}{\stepsize^2}\big( f(\y)-f(\x)-\innerProd{\nabla f(\x)}{\y-\x} \big) \ .\vspace{-1mm}
\end{equation}
Here the point $\overline\s$ is defined to be the point where the ray from $\x$ to the optimum $\x^*$ pinches the boundary of the set $\domain$, i.e. furthest away from $\x$ while still in $\domain$, %
$\overline\s(\x, \x^*,\domain) := \textrm{ray}(\x, \x^*) \cap \partial \domain$.
 We will later show that this strict interior assumption, which can be very prohibitive, can be removed for the Frank-Wolfe algorithm with \emph{away steps}, as we explain in Section~\ref{sec:convMFV}.
Clearly, the quantity $\strongConvFW$ is affine invariant, as it only depends on the inner products of feasible points with the gradient. %

\begin{remark}
For all pairs of functions $f$ and bounded sets $\domain$, it holds that $\strongConvFW \le \Cf$.
\end{remark}\vspace{-2mm}

The following simple lemma gives an interpretation of the very abstract (affine invariant) quantity defined above, in terms of classical norms and strong-convexity properties.

\begin{lemma}\label{lem:muFWinterpretation}
Let $f$ be a convex differentiable function and suppose $f$ is \emph{strongly convex} w.r.t. some arbitrary norm $\norm{.}$ over the domain $\domain$ with strong-convexity constant $\mu>0$.
Furthermore, suppose that the (unique) optimum $x^*$ lies in the relative interior of $\domain$, i.e. $\distToBoundary := \inf_{\s\in\partial \domain} \norm{\s-\x^*} > 0$.
Then\vspace{-3mm}
\[
\strongConvFW \geq \mu \cdot \distToBoundary^2  \ .
\] %
\end{lemma}

\vspace{-3mm}
\section{Linear Convergence of Frank-Wolfe}\vspace{-2mm}%
We obtain an affine invariant linear convergence proof for the standard FW algorithm when $f$ is strongly convex and the solution $\x^*$ lies in the relative interior of $\domain$ (an improvement over~\cite{Guelat:1986fq}). %

\begin{theorem}\label{thm:linearConvergenceFW}
Suppose that $f$ has smoothness constant $\Cf$ as defined in~(\ref{eq:Cf}),
as well as ``interior'' strong convexity constant~$\strongConvFW$ as defined in~(\ref{eq:mufFW}).
Then the error of the iterates of the Frank-Wolfe algorithm with step-size $\stepsize := \min\{1, \frac{g_k}{\Cf} \}$
(or using line-search) decreases geometrically, that is\vspace{-2mm}
\[
h_{k+1} \leq \left(1-\rho_f^\FW\right) h_k \ , \vspace{-3mm}
\]
where $\rho_f^\FW := \min \{\frac{1}{2}, \frac{\strongConvFW}{\Cf} \}$. Here in each iteration, $h_k := f(\x^{(k)}) - f(\x^*)$ denotes the primal error, and $g_k :=  g(\x^{(k)}) := \displaystyle\max_{\s \in \domain} \,\big\langle \nabla f(\x^{(k)}), \x^{(k)} - \s  \big\rangle$ is the duality gap as defined by \cite{Jaggi:2013wg}.%
\end{theorem}

\vspace{-2mm}
\section{Linear Convergence of Frank-Wolfe with Away-Steps}\vspace{-2mm}%
\label{sec:convMFV}

We now show the linear convergence of FW with away-steps under strong convexity, without any assumption on the location of the optimum with respect to the domain.
However, our convergence rate will depend on a purely geometric complexity constant of the domain $\domain$, as we show below.

\vspace{-2mm}
\paragraph{An Affine Invariant Notion of Strong Convexity which Depends on the Geometry of $\domain$.}

The trick is to use anchor points in the domain in order to define standard lengths (by looking at proportions on lines). These anchor points ($\s_f(\x)$ and $\vv_f(\x)$ defined below) are motivated directly from the away-steps algorithm.\\
Let $\s_f(\x) :=$%
$~\argmin_{\vv \in \Vertices} \left\langle \nabla f(\x), \vv \right\rangle$ (the standard Frank-Wolfe direction).
To define the away-vertex, we consider all possible expansions of $\x$ as a convex combination of vertices.
Let $\S_{\x} := \{ \S \, | \, \S \subseteq \Vertices$ such that $\x$ is a proper\footnote{By \emph{proper} convex combination, we mean that all coefficients are non-zero in the convex combination.} convex combination of all the elements in $\S\}$.
For a given set~$\S$, we write $\vv_{\S}(\x) := \argmax_{\vv \in \S } \left\langle \nabla f(\x), \vv \right\rangle$ for the away vertex in the algorithm supposing that the current set of active vertices was $\S$.
Finally, we define $\vv_f(\x) := \displaystyle\argmin_{\{\vv = \vv_{\S}(\x) \,|\, \S \in \S_{\x} \}} \textstyle\left\langle \nabla f(\x), \vv \right\rangle$ to be the worst-case away vertex (that is, the vertex which would yield the smallest away descent).

We can now define the strong convexity constant $\strongConvMFW$ which depends \emph{both} on the function~$f$ and the domain~$\domain$:\vspace{-3mm}
\begin{equation}\label{eq:muf}
  \strongConvMFW := \inf_{\x\in \domain} \inf_{\substack{\x^* \in \domain\\
                        \textrm{s.t. } \left\langle \nabla f(\x), \x^*-\x \right\rangle < 0 }}
           \frac{2}{{\stepsize^\away(\x,\x^*)}^2}
           \big( f(\x^*)-f(\x)-\left\langle \nabla f(\x),  \x^*-\x \right\rangle \big) \ .
\end{equation}\vspace{-4mm}

Here the positive quantity $\stepsize^\away(\x,\x^*) := \frac{\innerProd{\nabla f(\x)}{\x^*-\x} }{\innerProd{\nabla f(\x)}{\s_f(\x) - \vv_f(\x)} }$ plays the role of $\stepsize$ in definition~\eqref{eq:Cf}.

\vspace{-3mm}
\paragraph{Interpretation.}
The above complexity definition is already sufficient for us to prove the linear convergence. 
Additionally, the constant can be understood in terms of the geometry of $\domain$, as follows:

\vspace{-3mm}
\paragraph{Directional Width.}
The directional width of a set $\domain$ with respect to a direction $\dd$ (and underlying inner product norm $\norm{\cdot}$) is defined as $\dirW(\domain,\dd) := \max_{\x\in\domain} \big\langle \frac{\dd}{\dualnorm{\dd}}, \x \big\rangle-\min_{\x \in\domain} \big\langle \frac{\dd}{\dualnorm{\dd}}, \x \big\rangle$.\vspace{-2mm}

\vspace{-2mm}
\paragraph{Pyramidal Width.}
We define the pyramidal directional width of a set $\domain$ with respect to a direction $\dd$ and a base point $\x \in \domain$ to be
$
\PdirW(\domain,\dd, \x) := \min_{\S \in \S_{\x}} \dirW( \S \cup \{\s(\domain, \dd) \} , \; \dd)
$
where $\s(\domain,\dd) := \argmax_{\vv \in \domain}  \left\langle \dd, \vv \right\rangle$.
To define the pyramidal width of a set, we take the infimum over a set of possible feasible directions~$\dd$ (in order to avoid the problem of zero width).
A direction~$\dd$ is \emph{feasible} for $\domain$ from $\x$ if it points inwards the set, 
(i.e. $\dd \in \text{cone}(\domain-\x)$). 

We define the \emph{pyramidal width} of a set $\domain$ to be the smallest pyramidal width of all its faces, i.e.\vspace{-1mm}
\begin{equation} \label{eq:PdirW}
\PdirW(\domain) := \displaystyle \inf_{\substack{\Kface \in \textrm{faces}(\domain) \\
												  \x \in \Kface \\
												  \dd \in \text{cone}(\Kface-\x) \setminus \{\0\}} 
                                   } \PdirW(\Kface,\dd, \x) .    \vspace{-1mm}                             
\end{equation}

\begin{remark}
Any curved domain will yield a pyramidal width of zero, because then the set of active atoms $\S_{\x}$ can contain vertices arbitrary close to the boundary forming a very narrow pyramid. The pyramidal width quantity is thus only useful on polytopes (the convex hull of a finite set of points). %
\end{remark}

\begin{remark} \label{thm:PdirWproperty}
Let $ \vv(\x, \dd) :=$ the vertex which achieves the minimum in $\min_{\S \in \S_{\x}} \max_{\vv \in \S } \innerProd{\dd}{-\vv}$ for a polytope $\Kface$ and $\x \in \Kface$. Then we have
$
	\PdirW(\Kface,\dd, \x) = \big\langle \frac{\dd}{\dualnorm{\dd}} , \s(\Kface,\dd) - \vv(\x, \dd) \big\rangle.
$\vspace{-3mm}
\end{remark}
We conjecture that $\PdirW(\domain)$ for $\domain$ being the unit simplex in $\R^d$ is $\frac2{\sqrt{d}}$.

\begin{lemma} \label{lem:muFdirWinterpretation}
Let $f$ be a convex differentiable function and suppose that $f$ is $\mu$-\emph{strongly convex} w.r.t. some \emph{inner product norm} $\norm{\cdot}$ over the domain $\domain$ with strong-convexity constant $\mu \geq 0$. Then\vspace{-1mm}
\[
\strongConvMFW \geq \mu \cdot \left( \PdirW(\domain) \right)^2 \ .
\] 
\end{lemma}

\vspace{-2mm}

\begin{theorem}\label{thm:linear_convergenceMFW}
Suppose that $f$ has smoothness constant $\CfMFW$,\footnote{%
For a convenience in the proof, we use a slightly modified curvature constant $\CfMFW$, which is identical to the definition of $\Cf$, except that both positive and negative step-sizes are allowed, i.e. the range of $\stepsize$ in the definition for $\Cf$ is replaced by $[-1,1]$ instead of just $[0,1]$. 
Note that boundedness of this (again affine invariant) $\CfMFW$ is still implied by the Lipschitz continuity of the gradient of $f$ (over the slightly larger domain $\domain + \domain - \domain$, but with the same diameter constant).}%
as well as geometric strong convexity constant~$\strongConvMFW$ as defined in~(\ref{eq:muf}).
Then the error of the iterates of the FW algorithm with away-steps\footnote{%
In the algorithm, one can either use line-search or set the step-size as the feasible one that minimizes the quadratic upper bound given by the curvature $\CfMFW$%
, i.e. $\stepsize_k := \min\{1,\stepmax, \stepbound_k\}$ where $\stepbound_k := \frac{g_k}{2\CfMFW}$ and\vspace{-2mm} $g_k :=  \langle  -\nabla f(\x^{(k)}), \s_k - \vv_k \rangle$.
}
(Algorithm~\ref{alg:MFW})
decreases geometrically at each step that is not a drop step (i.e. when $\stepsize_k < \stepmax$), that is\vspace{-2mm}
\[
h_{k+1} \leq \left(1-\rho_f^\away\right) h_k \ ,\vspace{-2mm}
\]
where $\rho_f^\away := \frac{\strongConvMFW}{4\CfMFW}$. 
Moreover, the number of drop steps up to iteration $k$ is bounded by $k/2$. %
This yields the global linear convergence rate of $h_k \leq h_0 \exp(-\frac12 \rho_f^\away k)$.
\end{theorem}

\paragraph{Acknowledgements.}
Simon Lacoste-Julien acknowledges support by the ERC (SIERRA-ERC-239993).
Martin Jaggi acknowledges support
by the Simons Institute for the Theory of Computing,
by the Swiss National Science Foundation (SNSF), 
and by the ERC Project SIPA.

\bibliographystyle{alphaurl} %
{\small
\bibliography{bibliography}
}%

\newpage
\appendix

\section{Linear Convergence of Frank-Wolfe for Strongly Convex Functions with Optimum in the Interior}

\subsection{An Affine Invariant Notion of Strong Convexity}
We re-state and interpret the ``interior'' strong convexity constant~$\strongConvFW$ as defined in~(\ref{eq:mufFW}), that is
\[
  \strongConvFW := \inf_{\substack{\x\in \domain\setminus\{\x^*\}\\
                      \overline\s = \overline\s(\x, \x^*,\domain),\\%
                      \stepsize\in(0,1],\\
                      \y = \x+\stepsize(\overline\s-\x)}}
           \frac{2}{\stepsize^2}\big( f(\y)-f(\x)-\langle \nabla f(\x), \y-\x\rangle \big) \ .
\]
Here the point $\overline\s$ is defined to be the point where the ray from $\x$ to $\x^*$ pinches the boundary of the set $\domain$, %
i.e. $\overline\s(\x, \x^*,\domain) := \textrm{ray}(\x, \x^*) \cap \partial \domain$.

Recalling that the curvature $C_f$ by definition (\ref{eq:Cf}) provides an affine-invariant quadratic upper bound on the function $f$, the strong convexity constant $\strongConvFW$ here gives rise to an analogous quadratic lower bound, that is
\begin{equation}\label{eq:quadLowerBound}
f(\y)-f(\x)-\langle \nabla f(\x), \y-\x \rangle
\geq \frac{\stepsize^2}{2} \strongConvFW
\end{equation}
if the  point $\y = \x+\stepsize(\overline\s-\x)$ is determined by the boundary point $\overline\s(\x, \x^*,\domain)$ %
and an arbitrary step-size $\stepsize\in[0,1]$, i.e., if the point $\y$ lies on the segment which is between $\x$ and the boundary of~$\domain$, and passes through $\x^*$.
Here we prove Lemma \ref{lem:muFWinterpretation}, which gives a simple geometric interpretation of the abstract (affine invariant) quantity $\strongConvFW$ defined above, in terms of classical norms and strong-convexity properties.

\begin{replemma}{lem:muFWinterpretation}
Let $f$ be a convex differentiable function and suppose $f$ is \emph{strongly convex} w.r.t. some arbitrary norm $\norm{.}$ over the domain $\domain$ with strong-convexity constant $\mu>0$.

Furthermore, suppose that the (unique) optimum $x^*$ lies in the relative interior of $\domain$, i.e. $\distToBoundary := \inf_{\s\in\partial \domain} \norm{\s-\x^*} > 0$.
Then\vspace{-1mm}
\[
\strongConvFW \geq \mu \cdot \distToBoundary^2  \ .
\] %
\end{replemma}
\begin{proof}
By definition of strong convexity with respect to a norm, we have that for any $\x,\y\in\domain$,
\[
f(\y)- f(\x)- \langle\nabla f(\x), \y-\x \rangle
\geq \textstyle\frac{\mu}{2} \norm{\y-\x}^2 \ .
\]
We  want to use this lower bound in the definition (\ref{eq:mufFW}) of the affine invariant strong convexity constant. 
Observe that $\frac{1}{\stepsize^2} \norm{\y-\x}^2 = \norm{\overline\s-\x}^2$ for any $\x$ used in~\eqref{eq:mufFW}  since $\y := \x+\stepsize(\overline\s-\x) \in\domain$ by convexity. Moreover, by the definition of $\overline\s$ and $\distToBoundary$, $\norm{\overline\s-\x} \geq \norm{\overline\s-\x^*} \geq \distToBoundary$. Therefore, we can lower bound $\strongConvFW$ as
\[
\strongConvFW  
\geq \inf
          \textstyle \frac{2}{\stepsize^2}
          \frac{\mu}{2} \norm{\y-\x}^2
= \inf \mu \norm{\overline\s-\x}^2
\geq \mu \cdot \distToBoundary^2 \ ,
\]
which is the claimed bound.
\end{proof}

\subsection{Convergence Analysis}
\paragraph{Curvature.}
The definition of the curvature constant $\Cf$ as in (\ref{eq:Cf}) directly gives an affine invariant quadratic upper bound on the objective function, as follows:

Let $\x_\stepsize :=  \x + \stepsize (\s-\x)$ be the point obtained by moving with step-size $\stepsize$ in direction $\s\in\domain$. By definition of $\Cf$, we have
\[\textstyle
f(\x_\stepsize) ~\leq~ f(\x) + \stepsize \left\langle  \nabla f(\x), \s-\x \right\rangle + \frac{\stepsize^2}{2} \Cf,~~~~~\forall \stepsize \in [0,1] \ .
\]
This crucial bound enables us to analyze the objective improvement in each iteration in Frank-Wolfe-type algorithms, as in \cite{Jaggi:2013wg}:
If the point $\s$ is the standard Frank-Wolfe direction returned by an exact linear oracle, then the middle quantity is exactly the negative of the duality gap, $\left\langle  \nabla f(\x), \s-\x \right\rangle = -g(\x)$.
If an inexact linear oracle is used instead, which has multiplicative approximation quality~$\mapprox$ (to be defined below), then we always have the upper bound
\begin{equation}\label{eq:CfUpperBound}\textstyle
f(\x_\stepsize) ~\leq~ f(\x) - \stepsize \mapprox g(\x) + \frac{\stepsize^2}{2} \Cf
,~~~~~\forall \stepsize \in [0,1] \ .
\end{equation}

\paragraph{Inexact Linear Oracles.}
The standard linear oracle used inside the classical Frank-Wolfe algorithm is given by $\s \in \argmin_{\vv \in \domain} \left\langle \nabla f(\x), \vv \right\rangle$.
We say that the linear oracle %
satisfies multiplicative accuracy~$\mapprox$ for some $\mapprox\in[0,1]$, if for any $\x\in\domain$, the returned $\s$ is such that
\begin{equation}\label{eq:qualityMult}
\left\langle \nabla f(\x),  \x-\s  \right\rangle
~\ge~
\displaystyle \mapprox\cdot \max_{\s' \in \domain} \textstyle\left\langle \nabla f(\x), \x-\s'  \right\rangle
\ .
\end{equation}
Note that the classical Frank-Wolfe direction $\s$ satisfies this inequality with $\mapprox=1$.
The inequality means that the oracle answer $\s$ attains at least a $\mapprox$-fraction of the current \emph{duality gap} $g(\x) := \displaystyle\max_{\s \in \domain} \,\big\langle \nabla f(\x^{(k)}), \x - \s  \big\rangle$ as defined by \cite{Jaggi:2013wg}.

\emph{Related work.} The sublinear convergence of Frank-Wolfe with $O(1/k)$ is known to also hold if this linear subproblems are only solved approximately (meaning that the linear oracle is inexact).
For additive approximation accuracy, this was shown by \cite{Dunn:1978di,Dunn:1979da} for the line-search case, and by \cite{Jaggi:2011ux,Jaggi:2013wg} for the simpler $\frac2{k+2}$ step-size and the primal-dual convergence. For multiplicative accuracy (relative to the duality gap), it was shown by \cite[Appendix C]{LacosteJulien:2013ue}.
The case of the inexact or noisy gradient information can also be analyzed in the same way, as discussed in $\cite{Jaggi:2013wg,Freund:2013uk}$.

\paragraph{Linear Convergence Proof.}
Here we prove a slightly stronger version of Theorem \ref{thm:linearConvergenceFW}, showing the linear convergence also in the case where the linear subproblems in each iteration are only solved approximately.
The exact oracle case is obtained for $\mapprox :=1$.

\begin{reptheorem}{thm:linearConvergenceFW}
Suppose that $f$ has smoothness constant $\Cf$ as defined in~(\ref{eq:Cf}),
as well as ``interior'' strong convexity constant~$\strongConvFW$ as defined in~(\ref{eq:mufFW}).

Then the error of the iterates of the Frank-Wolfe algorithm with step-size $\stepsize := \min\{1, \frac{\mapprox g_k}{\Cf} \}$
(or using line-search) decreases geometrically, that is\vspace{-1mm}
\[
h_{k+1} \leq \left(1-\rho_f^\FW\right) h_k \ , \vspace{-2mm}
\]
where $\rho_f^\FW := \min \{\frac\mapprox2, \mapprox^2 \frac{\strongConvFW}{\Cf} \}$. Here in each iteration, $h_k := f(\x^{(k)}) - f(\x^*)$ denotes the primal error, and $g_k :=  g(\x^{(k)}) := \displaystyle\max_{\s \in \domain} \,\big\langle \nabla f(\x^{(k)}), \x^{(k)} - \s  \big\rangle$ is the duality gap as defined by \cite{Jaggi:2013wg}, 
and $\mapprox\in[0,1]$ is the multiplicative approximation quality to which the linear sub-problems are solved.
\end{reptheorem}
\begin{proof}
Applying the strong convexity bound (\ref{eq:quadLowerBound}) at the current iterate $\x:=\x^{(k)}$ for the special step-size $\overline\stepsize$ such that $\y = \x^{(k)}+\overline\stepsize (\overline\s-\x^{(k)}) = \x^*$ gives
\[
\begin{array}{rll}
\frac{\overline\stepsize^2}{2} \strongConvFW \  \leq
&f(\y)-f(\x) -\langle \nabla f(\x), \y-\x \rangle \\
=&f(\x^*)-f(\x^{(k)}) - \overline\stepsize \left\langle \nabla f(\x^{(k)}), \overline\s-\x^{(k)}\right\rangle \\
\le &-h_k + \overline\stepsize  g_k \ .
\end{array}
\]
Therefore $h_k \le -\frac{\overline\stepsize^2}{2} \strongConvFW + \overline\stepsize g_k$,
which is upper bounded by $%
\frac{{g_k}^2}{2\strongConvFW}$.\\
{\small(Here we have used the trivial inequality $0 \le %
a^2-2ab+b^2$ for the choice of numbers $a:=\frac{g_k}{\strongConvFW}$ and $b:=\overline\stepsize$)}

We now want to use the curvature definition to lower bound the absolute progress~$h_k - h_{k+1}$. The definition of the curvature $\Cf$ in the form of the quadratic upper bound \eqref{eq:CfUpperBound} reads as %
$h_k - h_{k+1} \ge \stepsize \mapprox g_k - \frac{\stepsize^2}{2} \Cf$. Using this for the particular step-size $\stepsize:=\frac{\mapprox g_k}{\Cf}$, the r.h.s. is $= \frac{\mapprox^2 g_k^2}{2\Cf}$. (The border case when $\frac{\mapprox g_k}{\Cf}>1$ will be discussed separately below). The same inequality also holds in the line-search case, as the improvement only gets better.
Combining the two bounds, we have obtained
\[
\frac{h_k - h_{k+1} }{h_k} \ge \mapprox^2 \frac{\strongConvFW}{\Cf}\vspace{-1mm}
\]
implying that we have a geometric rate of decrease $h_{k+1} \leq \Big(1-\mapprox^2\frac{\strongConvFW}{\Cf}\Big) h_k$.

\emph{Border case.} In the above analysis, we have assumed that the step-size $\stepsize:=\frac{\mapprox g_k}{\Cf} \le 1$. If this is not the case (i.e. if $\mapprox g_k > \Cf$), then the actual step-size in the algorithm is clipped to $1$, in which case the curvature upper bound~\eqref{eq:CfUpperBound} for $\stepsize:=1$ gives 
$h_k - h_{k+1} 
~\geq~ \mapprox g_k - \frac12 \Cf 
> \mapprox g_k - \frac12 \mapprox g_k
= \frac\mapprox2 g_k$. Using that the main property $h_k \leq g_k$ of the duality gap (by convexity), we therefore have $\frac{h_k - h_{k+1}}{h_k} > \frac\mapprox2$, which gives a geometric decrease of the error with constant $1-\frac\mapprox2$.
\end{proof}

\section{Linear Convergence of FW with Away-Steps under Strong Convexity}
\subsection{Interpretation of the Geometric Strong Convexity Constant $\strongConvMFW$}
The geometric strong convexity constant $\strongConvMFW$, as defined in (\ref{eq:muf}), is affine invariant, since it only depends on the inner products of feasible points with the gradient. Also, it combines both the complexity of the function~$f$ and the geometry of the domain $\domain$.
The goal of this subsection is to prove Lemma~\ref{lem:muFdirWinterpretation}, which provides a geometric interpretation of $\strongConvMFW$. The lemma allows us to bound the constant $\strongConvMFW$ in terms of the strong convexity of the objective function, combined with a purely geometric complexity measure of the domain $\domain$. In the following Section \ref{sec:MFWconv} below, we will show the linear convergence of Algorithm~\ref{alg:MFW} under the assumption that $\strongConvMFW > 0$. From the view of Lemma~\ref{lem:muFdirWinterpretation}, $\strongConvMFW > 0$ is a slightly weaker condition than the strong convexity of the function over a polytope domain (it is implied by strong convexity).

We recall the definition of $\strongConvMFW$ as given in (\ref{eq:muf}):
\[
  \strongConvMFW := \inf_{\x\in \domain} \inf_{\substack{\x^* \in \domain\\
                        \textrm{s.t. } \left\langle \nabla f(\x), \x^*-\x \right\rangle < 0 }}
           \frac{2}{{\stepsize^\away(\x,\x^*)}^2}
           \big( f(\x^*)-f(\x)-\left\langle \nabla f(\x),  \x^*-\x \right\rangle \big) \ .
\]\vspace{-2mm}

Here the positive quantity $\stepsize^\away(\x,\x^*) := \frac{\innerProd{\nabla f(\x)}{\x^*-\x} }{\innerProd{\nabla f(\x)}{\s_f(\x) - \vv_f(\x)} }$ plays the role of $\stepsize$ in the analogous upper bound definition~\eqref{eq:Cf} for the curvature. We recall that $\s_f(\x) := \argmin_{\vv \in \Vertices} \left\langle \nabla f(\x), \vv \right\rangle$ and that $\vv_f(\x) := \displaystyle\argmin_{\{\vv = \vv_{\S}(\x) \,|\, \S \in \S_{\x} \}} \textstyle\left\langle \nabla f(\x), \vv \right\rangle$. 

We recall the definition of the pyramidal directional width of a set $\domain$ with respect to a direction $\dd$ and a base point $\x \in \domain$:
$
\PdirW(\domain,\dd, \x) := \min_{\S \in \S_{\x}} \dirW( \S \cup \{\s(\domain, \dd) \} , \; \dd)
$
where $\s(\domain,\dd) := \argmax_{\vv \in \domain}  \left\langle \dd, \vv \right\rangle$. We now provide a proof of Remark~\ref{thm:PdirWproperty}, which will be useful at the end of the proof of Lemma~\ref{lem:muFdirWinterpretation}.

\begin{repremark}{thm:PdirWproperty}
Let $ \vv(\x, \dd) :=$ the vertex which achieves the minimizer of  $\min_{\S \in \S_{\x}} \max_{\vv \in \S } \innerProd{\dd}{-\vv}$ for a polytope $\Kface$ and $\x \in \Kface$. Then we have
\begin{equation} \label{eq:PdirWproperty}
	\PdirW(\Kface,\dd, \x) = \innerProd{\frac{\dd}{\dualnorm{\dd}}}{ \s(\Kface,\dd) - \vv(\x, \dd)}.
\end{equation}
\end{repremark}
\begin{proof}
\begin{align*}
\PdirW(\Kface,\dd, \x) &= \frac{1}{\norm{\dd}_*} \min_{\S \in \S_{\x}} \left( \max_{\y \in  \S \cup \{\s(\Kface,\dd) \}} \innerProd{\dd}{\y} - \min_{\y \in \S \cup \{\s(\Kface,\dd) \} } \innerProd{\dd}{\y} \right) \\
	&= \frac{1}{\norm{\dd}_*} \min_{\S \in \S_{\x}} \left( \innerProd{\dd}{\s(\Kface,\dd)} + \max_{\y \in \S} \innerProd{\dd}{-\y} \right) \\
	&= \frac{1}{\norm{\dd}_*} \left( \innerProd{\dd}{\s(\Kface,\dd)} +  \min_{\S \in \S_{\x}} \max_{\y \in \S} \innerProd{\dd}{-\y} \right) \\
	&=\innerProd{\frac{\dd}{\dualnorm{\dd}}}{ \s(\Kface,\dd) - \vv(\x, \dd)}.
	\vspace{-5mm}
\end{align*}
\end{proof}

\paragraph{Exposing a facet of a polytope.} Finally, we introduce a final concept that will be useful in the proof. We say that a direction $\dd$ \textbf{exposes a facet}\footnote{As a reminder, we define a \emph{k-face} of $\domain$ (a $k$-dimensional face of $\domain$) a set $\Kface$ such that $\Kface = \domain \cap \{ \y : \innerProd{\r}{\y - \x} = \0 \}$ for some normal vector $\r$ and fixed reference point $\x \in \Kface$ with the additional property that $\domain$ lies on one side of the given half-space determined by $\r$ i.e. $ \innerProd{\r}{\y - \x} \leq \0$ $\forall \y \in \domain$. $k$ is the dimensionality of the affine hull of $\Kface$. We call a $k$-face of dimensions $k = 0$, $1$, $\textrm{dim}(\domain)-2$ and $\textrm{dim}(\domain)-1$ a \emph{vertex}, \emph{edge}, \emph{ridge} and \emph{facet} respectively. $\domain$ is a $k$-face of itself with $k = \textrm{dim}(\domain)$. See definition 2.1 in~\cite{Ziegler:1995td}.}  $\mathcal{F}$ of the polytope $\domain$ \textbf{at $\x$} if 1) $\mathcal{F}$ includes $\x$ and is a facet of $\domain$; and 2) the orthogonal component of $\dd$ to this facet defines this facet with $\dd$ on one side and $\domain - \x$ on the other side. In other words, let $\mathcal{F}_{\s} := \textrm{span}(\domain - \x)$ be the affine hull of $\mathcal{F}$ re-centered at $\x$; let $\proj_{\mathcal{F}_{\s}}$ be the orthogonal projection operator onto $\mathcal{F}_{\s}$; then the second condition can be expressed as $\mathcal{F} = \{\y \in \domain : \innerProd{(\id - \proj_{\mathcal{F}_{\s}})\dd}{\y-\x} = 0 \}$ and $\innerProd{(\id - \proj_{\mathcal{F}_{\s}})\dd}{\y-\x} \leq 0$ $\forall \y \in \domain$ (note that $(\id - \proj_{\mathcal{F}_{\s}})\dd$ is the orthogonal component of $\dd$ to the facet $\mathcal{F}$). Note that these conditions imply that $\dd$ cannot be a feasible direction, i.e. $\dd \notin \textrm{cone}(\domain-\x)$ and that $\x$ must be on the (relative) boundary of $\domain$. It turns out that the converse is also true: if $\dd \notin \textrm{cone}(\domain-\x)$, then there must exist at least a facet of $\domain$ exposed by $\dd$ at $\x$.\footnote{To find such an exposed facet, consider the $\mathcal{H}$-polyhedron representation of $\textrm{cone}(\domain-\x)$ (see~\cite{Ziegler:1995td}). As $\dd$ is not feasible, at least one halfspace constraint must be violated; the intersection of the hyperplane determining this halfspace constraint with $\domain-\x$ yields (the translation of) one exposed facet.}

\begin{replemma}{lem:muFdirWinterpretation}
Let $f$ be a convex differentiable function and suppose that $f$ is $\mu$-\emph{strongly convex} w.r.t. some \emph{inner product norm} $\norm{\cdot}$ over the domain $\domain$ with strong-convexity constant $\mu \geq 0$. Then
\[
\strongConvMFW \geq \mu \cdot \left( \PdirW(\domain) \right)^2 \ .
\] 
\end{replemma}
\begin{proof}
By definition of strong convexity with respect to a norm, we have that for any $\x,\y\in\domain$,
\begin{equation} \label{eq:strongConv}
f(\y)- f(\x)- \langle\nabla f(\x), \y-\x \rangle
\geq \textstyle\frac{\mu}{2} \norm{\y-\x}^2 \ .
\end{equation}
Using the strong convexity bound~\eqref{eq:strongConv} with $\y := \x^*$ on the right hand side of equation~\eqref{eq:muf} (and using the shorthand $\r_{\x} := -\nabla f(\x)$ ), we thus get:
\begin{align}
\strongConvMFW \geq&  \inf_{\substack{\x, \x^* \in \domain\\
                                   \textrm{s.t. } \innerProd{\r_{\x}}{\x^*-\x} > 0}}
                      \mu \left(  \frac{\innerProd{\r_{\x}}{ \s_f(\x) - \vv_f(\x)}}{\innerProd{\r_{\x}}{\x^*-\x}} \norm{{\x^*-\x}} \right)^2 \nonumber \\
		&=  \mu \inf_{\substack{\x \neq \x^* \in \domain \\
                        \textrm{s.t. } \innerProd{\r_{\x}}{ \hat{\r}_{\x,\x^*}} > 0}}
           \left(  \frac{\innerProd{\r_{\x}}{ \s_f(\x) - \vv_f(\x)}}{\innerProd{\r_{\x}}{ \hat{\r}_{\x,\x^*}}} \right)^2  , \label{eq:mufInitial}
\end{align}
where $\hat{\r}_{\x,\x^*} := \frac{\x^*-\x}{\norm{\x^*-\x}}$ is the unit norm feasible direction from $\x$ to $\x^*$. We are thus taking an infimum over all possible feasible directions starting from $\x$ (i.e. which moves within $\domain$) with the additional constraint that it makes a positive inner product with the negative gradient $\r_{\x}$ i.e. it is a strict descent direction. This is only possible if $\x$ is not already optimal, i.e. $\x \in \domain \setminus \X^*$ where $\X^* := \{\x^* \in \domain : \innerProd{\r_{\x^*}}{\x-\x^*} \leq 0 \,\, \forall \x \in \domain \}$ is the set of optimal points. {\small[NOTE: I know that by strong convexity it only contains one point; but I wanted to keep it general here just to see the effect of the constraints and to get more intuition about the constants]}.

The goal in the rest of the proof is to equivalently project $\r_{\x}$ onto facets of $\domain$ and then to characterize the property of its projection so that we can consider a wider set of valid directions that will thus yield a lower bound on the infimum of $\frac{\innerProd{\r_{\x}}{ \s_f(\x) - \vv_f(\x)}}{\innerProd{\r_{\x}}{ \hat{\r}_{\x,\x^*}}}$. For the rest of the proof, we fix $\x \notin \X^*$ and we work on the centered polytope at $\x$ i.e. let $\tilde{\domain} = \domain - \x$. During the proof, we work on faces $\Kface_l$ of $\tilde{\domain}$ of decreasing dimensions which all include $\x$ at their origin, as well as maintain a projection of the gradient $\dd_l \in \C_l := \text{span}(\Kface_l)$. We let $\proj_l$ be the orthogonal projection operator onto $\C_l$. We will keep projecting the gradient as $\dd_l := \proj_l \dd_{l-1}$ until they become a non-zero feasible direction from the origin i.e. $\dd_l \in \text{cone}(\Kface_l) \setminus \{\0\}$, at which point we will exit the loop with $\dd = \dd_l$ and $\Kface = \Kface_l$ for the last considered face. 

We start with $\Kface_0 = \tilde{\domain}$ and we note that since both $\s_f(\x) - \vv_f(\x)$ and $\hat{\r}_{\x,\x^*}$ belong to $\C_0 = \text{span}(\Kface_0)$, if we let $\dd_0 = \proj_0 \r_{\x}$, then we have $\frac{\innerProd{\r_{\x}}{ \s_f(\x) - \vv_f(\x)}}{\innerProd{\r_{\x}}{ \hat{\r}_{\x,\x^*}}} = \frac{\innerProd{\dd_0}{ \s_f(\x) - \vv_f(\x)}}{\innerProd{\dd_0}{ \hat{\r}_{\x,\x^*}}}$ for any $\x^*$ such that $\innerProd{\r_{\x}}{\x^*-\x} \neq 0$. Then we consider whether $\dd_0$ is a feasible direction in $\Kface_0$. If $\dd_0$ is feasible i.e. $\dd_0 \in \text{cone}(\Kface_0)$, then we stop with $\dd = \dd_0 = \proj_0 \r_{\x}$ and $\Kface = \Kface_0$. By the definition of the dual norm $\dualnorm{\cdot}$ (generalized Cauchy-Schwartz), we have $\innerProd{\dd}{ \hat{\r}_{\x,\x^*}} \leq \dualnorm{\dd}\norm{\hat{\r}_{\x,\x^*}} = \dualnorm{\dd} \cdot 1$, and thus for this $\x$ we have:
\[ 
	\inf_{\substack{\x^* \in \domain\\
               \textrm{s.t. } \innerProd{\r_{\x}}{ \hat{\r}_{\x,\x^*}} > 0}}
                       \frac{\innerProd{\r_{\x}}{ \s_f(\x) - \vv_f(\x)}}{\innerProd{\r_{\x}}{ \hat{\r}_{\x,\x^*}}} \geq  \innerProd{\frac{\dd_0}{\dualnorm{\dd_0}}}{ \s_f(\x) - \vv_f(\x)} .
\]
In the other possibility ($\dd_0 \notin \text{cone}(\Kface_0)$), then there must exist a least one facet $\Kface_1$ of $\Kface_0$ that is exposed by $\dd_0$ at $\0$ (note that we cannot have $\dd_0 = \0$ since $\x \notin \X^*$). %
We now project $\dd_0$ on $\text{span}(\Kface_1)$: $\dd_1 := \proj_1 \dd_0$, and we show how the lower bound transforms.
This yields the following inequalities:
\begin{align}
\innerProd{\r_{\x}}{ \s_f(\x) - \vv_f(\x)}  &= 
	 \max_{\s \in \domain} \innerProd{\r_{\x}}{ \s - \x} +
	 \min_{\S \in \S_{\x}} \max_{\vv \in \S} \innerProd{-\r_{\x}}{ \vv - \x}  \nonumber \\
	 &= \max_{\y \in \Kface_0} \innerProd{\dd_0}{ \y } +
	 	 \min_{\S \in \S_{\x}} \max_{\vv \in \S} \innerProd{-\dd_0}{ \vv - \x} \nonumber \\
	 &\ge \max_{\y \in \Kface_1} \innerProd{\dd_0}{ \y } +
	 	 	 \min_{\S \in \S_{\x}} \max_{\vv \in \S \cap (\Kface_1 + \x)} \innerProd{-\dd_0}{ \vv - \x} \nonumber \\
	 &= \max_{\y \in \Kface_1} \innerProd{\dd_1}{ \y } +
	 	 	 \min_{\S \in \S_{\x}} \max_{\vv \in \S } \innerProd{-\dd_1}{ \vv - \x} \nonumber \\
	 &= \innerProd{\dd_1}{\s(\Kface_1, \dd_1)} + \innerProd{-\dd_1}{ \vv(\x, \dd_1) - \x} . \label{eq:d1sf}
\end{align}
From the first to the second line, we used the fact that $\innerProd{\r_{\x} - \dd_0}{\y} = 0$ for any $\y \in \Kface_0 = \domain - \x$ as $\dd_0$ is the orthogonal projection of $\r_{\x}$ on $\C_0 = \text{span}(\Kface_0)$ (and thus we also have that $\innerProd{\r_{\x}}{\s_f(\x) - \x} = \innerProd{\dd_0}{\s(\Kface_0, \dd_0)}$). To go from the second to the third line, we use the fact that the first term yields an inequality as $\Kface_1 \subseteq \Kface_0$. Also, let $\Kface_{\x}$ be the minimal dimensional face of $\domain$ containing $\x$ (and thus $\x$ is in the relative interior of $\Kface_{\x}$). Note that $\bigcup \S_{\x} = \text{vertices}(\Kface_{\x})$, and also that $\Kface_{\x}$ is included in any other face containing $\x$. We thus have $\S \subseteq \Kface_1 + \x$ for any $\S \in \S_{\x}$ and thus the second term on the second line yielded an equality. The fourth line used the fact that $\dd_0 - \dd_1$ is orthogonal to members of $\Kface_1$. The fifth line used the definition of $\s(\Kface_1, \dd_1)$ and introduced the notation $ \vv(\x, \dd) :=$ the vertex $\vv \in \Kface_{\x}$ which achieves the minimizer of $\min_{\S \in \S_{\x}} \max_{\vv \in \S } \innerProd{\dd}{ -\vv}$.

To deal with $\innerProd{\r_{\x}}{ \hat{\r}_{\x,\x^*}} = \innerProd{\dd_0}{ \hat{\r}_{\x,\x^*}}$, we use the crucial fact that $\dd_0$ \emph{exposes the facet} $\Kface_1$ of $\Kface_0$. This implies that $\innerProd{\dd_0-\proj_0 \dd_0}{ \hat{\r}_{\x,\x^*}} \leq 0$ for all $\x^*-\x \in \Kface_0 \setminus \{\0\}$. So consider $\r_0 := \displaystyle \argmax_{\substack{\y \in \Kface_0 \\ 
				 \innerProd{\dd_0}{\y} > 0}} \innerProd{\dd_0}{ \frac{\y}{\norm{\y}}}$. 
We claim that we can choose $\r_0 \in \Kface_1$. To see this, let $\r_1 = \proj_1 \r_0$ and write $\r_1^\perp = \r_0 - \r_1$ and $\dd_1^\perp = \dd_0 - \dd_1$. Then we have:
\begin{align}
  \innerProd{\dd_0}{\frac{\r_0}{\norm{\r_0} }}  &= \frac{1}{\norm{\r_0}} \innerProd{\dd_1 + \dd_1^\perp}{\r_1 + \r_1^\perp} \nonumber \\
	&= \frac{1}{\norm{\r_0}} \big( \innerProd{\dd_1}{\r_1} + 0 + \underbrace{\innerProd{\dd_1^\perp}{\r_1+\r_1^\perp}}_{\leq 0} \big) \nonumber \\
	&\leq  \frac{1}{\norm{\r_0}} \innerProd{\dd_1}{\r_1} \leq  \frac{1}{\norm{\r_1}} \innerProd{\dd_1}{\r_1}, \text{ and thus} , \nonumber \\
\max_{\substack{\y \in \Kface_0 \\ 
				 \innerProd{\dd_0}{\y} > 0}} \innerProd{\dd_0}{ \frac{\y}{\norm{\y}}} &= 
		\max_{\substack{\y \in \Kface_1 \\ 
						 \innerProd{\dd_1}{\y} > 0}} \innerProd{\dd_1}{ \frac{\y}{\norm{\y}}} .		 
 \label{eq:d1r1}
\end{align}
Note that in the third line, we have used that $\norm{\r_1} = \norm{\proj_1 \r_0 - \proj_1 \0} \leq \norm{\r_0 - \0}$ by the contraction property of the orthogonal projection for inner product norms.\footnote{The contraction property is only valid for inner product norms (i.e. $\norm{\cdot} = \sqrt{ \innerProd{\cdot}{\cdot}}$), so this is where the assumption that the norm was generated by an inner product comes into play. 
} In the last line, we have an equality instead of the $\leq$ inequality as $\innerProd{\dd_1}{\y} = \innerProd{\dd_0}{\y}$ $\forall \y \in \Kface_1$ and $\Kface_1 \subseteq \Kface_0$, and so we also have the $\geq$ direction. Combining the facts from~\eqref{eq:d1sf} and~\eqref{eq:d1r1}, we get in this case:
\begin{align*}
	\inf_{\substack{\x^* \in \domain\\
               \textrm{s.t. } \innerProd{\r_{\x}}{ \hat{\r}_{\x,\x^*}} > 0}}
                       \frac{\innerProd{\r_{\x}}{ \s_f(\x) - \vv_f(\x)}}{\innerProd{\r_{\x}}{ \hat{\r}_{\x,\x^*}}} 
                       &\geq  \innerProd{\dd_1}{\s(\Kface_1, \dd_1)+\x - \vv(\x, \dd_1)} 	\left( \max_{\substack{\y \in \Kface_1 \\ 
                       						 \innerProd{\dd_1}{\y} > 0}} \innerProd{\dd_1}{ \frac{\y}{\norm{\y}}} \right)^{-1} 
\end{align*}

We are now back to a similar situation as before, but with $\Kface_1$ instead of $\Kface_0$ as the reference polytope. Note that by the third line of~\eqref{eq:d1r1}, we have $\innerProd{\dd_1}{\r_1} \geq \innerProd{\dd_0}{\r_0} > 0$ and thus $\dd_1 \neq \0$ (which is crucial to avoid a trivial lower bound of zero). So again, we consider whether $\dd_1 \in \text{cone}(\Kface_1)$. If $\dd_1 \in \text{cone}(\Kface_1)$, we stop here with $\dd = \dd_1$ and $\Kface = \Kface_1$. By Cauchy-Schwartz, we again have $\displaystyle 	\max_{\substack{\y \in \Kface \\ \innerProd{\dd}{\y} > 0}} \innerProd{\dd}{ \frac{\y}{\norm{\y}}} \leq \norm{\dd}_*$, and so we conclude
\begin{align} \label{eq:muCSlower}
	\inf_{\substack{\x^* \in \domain\\
               \textrm{s.t. } \innerProd{\r_{\x}}{ \hat{\r}_{\x,\x^*}} > 0}}
                       \frac{\innerProd{\r_{\x}}{ \s_f(\x) - \vv_f(\x)}}{\innerProd{\r_{\x}}{ \hat{\r}_{\x,\x^*}}} 
                       &\geq  \innerProd{\frac{\dd}{\norm{\dd}_*}}{\s(\Kface+\x, \dd) - \vv(\x, \dd)} 	
\end{align}
where $\dd \in \text{cone}(\Kface) \setminus \{\0\}$.

If $\dd_1 \notin \text{cone}(\Kface_1)$, then we continue our iterative process: we get that $\dd_1$ exposes a facet $\Kface_2$ of $\Kface_1$. We thus project $\dd_1$ on $\Kface_2$ to get $\dd_2 = \proj_2 \dd_1$. We can repeat exactly the same argument as before to get~\eqref{eq:d1sf} and~\eqref{eq:d1r1} with $\dd_2$ and $\Kface_2$ in place of $\dd_1$ and $\Kface_1$ (and $\dd_2 \neq \0$). If $\dd_2 \in \text{cone}(\Kface_2)$, then we stop with $\dd = \dd_2$ and $\Kface = \Kface_2$ and we again get the inequality~\eqref{eq:muCSlower}. Otherwise, we get an exposed facet $\Kface_3$, and repeat the process with $\dd_3 = \proj_3 \dd_2$. This process must stop at some point $l$: at the latest, we will reach $\Kface_l = \Kface_{\x}-\x$, the minimal dimensional face containing $\0$. In this case we must have $\dd_l \in \text{cone}(\Kface_l)$ as $\0$ is in the relative interior of $\Kface_l$ for a minimal face and so all directions are feasible.
We also note that $\dd_l \neq \0$ by the argument in~\eqref{eq:d1r1} that implies $\innerProd{\dd_l}{s(\Kface_l, \dd_l)} > 0$  (this condition is crucial to avoid having a lower bound of zero!). The latter also implies that the dimensionality of $\Kface_l$ must at least be 1. Letting again $\dd = \dd_l$ and $\Kface = \Kface_l$, we get inequality~\eqref{eq:muCSlower} with $\dd \in \text{cone}(\Kface) \setminus \{\0\}$.

From this argument, we can see that by considering all the possible faces of $\tilde{\domain}$ of dimension at least one which includes $\0$, and any feasible directions for these faces, we are sure to include the $\dd$ and~$\Kface$ that appears in \eqref{eq:muCSlower}. Translating back to the affine space $\domain$ (i.e. we use $\Kface+\x$ as the face of $\domain$ which contains $\x$), we can start to vary $\x$ again. We thus obtain the following lower bound:
\begin{align*}
	\inf_{\x \notin \X^*} \inf_{\substack{\x^* \in \domain\\
               \textrm{s.t. } \innerProd{\r_{\x}}{ \hat{\r}_{\x,\x^*}} > 0}}
                       \frac{\innerProd{\r_{\x}}{ \s_f(\x) - \vv_f(\x)}}{\innerProd{\r_{\x}}{ \hat{\r}_{\x,\x^*}}} 
          &\geq \inf_{\x \notin \X^*} \inf_{\substack{\Kface \in \textrm{faces}(\domain) \\
          												  \Kface \ni \x \\
          												  \dd \in \text{cone}(\Kface-\x) \setminus \{\0\} }}       												  
           \innerProd{\frac{\dd}{\dualnorm{\dd}}}{ \s(\Kface,\dd) - \vv(\x, \dd)} \\
          &\geq  \inf_{\substack{\Kface \in \textrm{faces}(\domain) \\
                    								   \x \in \Kface \\
                    								 \dd \in \text{cone}(\Kface-\x) \setminus \{\0\} }}
                    \PdirW(\Kface,\dd, \x) 
          = \PdirW(\domain) .
\end{align*}
For the last inequality, we used~\eqref{eq:PdirWproperty} from Remark~\ref{thm:PdirWproperty}. Combining this statement with~\eqref{eq:mufInitial} concludes the proof. 
\end{proof}

\subsection{Linear Convergence Proof}\label{sec:MFWconv}

\paragraph{Curvature Constants.}
Because of the additional possibility of the away step in Algorithm~\ref{alg:MFW}, we need to define the following slightly modified additional curvature constant, which will be needed for the linear convergence analysis of the algorithm\footnote{This can be avoided if the algorithm uses the step-size that minimizes a quadratic upper bound (see the proof for Theorem~\ref{thm:linear_convergenceMFW}; we can actually use $\stepsize_k := \min\{1,\stepmax, \frac{g_k}{2\Cf}\}$); but then one needs to compute an upper bound on $\Cf$ to run the algorithm (which is not always easy). Moreover, this algorithm might have less chance to get the `best case' behavior by being less adaptive.}:
\begin{equation}\label{eq:CfA}
  \CfA := \sup_{\substack{\x,\s\in \domain, \\
                      \stepsize\in[0,1],\\
                      \y = \x+\stepsize(\x-\s)}}
           \frac{2}{\stepsize^2}\big( f(\y)-f(\x)-\langle \nabla f(\x), \y-\x \rangle \big) \ .
\end{equation}
By comparing with $\Cf$~\eqref{eq:Cf}, we see that the modification is that $\y$ is defined with the \emph{away} direction $\x -\s$ instead of a standard FW direction $\s - \x$. This might yield some $\y$'s which are outside of the domain $\domain$ (in fact, $\y \in \domain^\away := \domain + (\domain - \domain)$ in the Minkowski sense). On the other hand, by re-using a similar argument as in \cite[Lemma 7]{Jaggi:2013wg}, we can obtain the same bound~\eqref{eq:CfBound} for $\CfA$, with the only difference that the Lipschitz constant $L$ for the gradient function has to be valid on $\domain^\away$ instead of just $\domain$.
Finally, the curvature constant for Algorithm~\ref{alg:MFW} is simply the worst-case possibility between the standard FW steps and the away steps:
\begin{equation}\label{eq:CfMFW}
  \CfMFW := \max\{\Cf, \CfA\}.
\end{equation}

\begin{remark}\label{rem:mfMFWsmallerThanCf}
For all pairs of functions $f$ and domains $\domain$, it holds that $\strongConvMFW \le \Cf$ (and $\Cf \le \CfMFW$).
\end{remark}\vspace{-2mm}
\begin{proof}
Choose $\x^* := \s_f(\x)$ for an $\x$ that is an away corner (i.e. $\x = \vv_f(\x)$) in~\eqref{eq:muf}. Then $\stepsize^\away(\x,\x^*) = 1$ and so we have $\y := \x^* = \x + \stepsize (\x^*-\x)$ with $\stepsize = 1$ which can also be used in the definition of $\Cf$. Thus, we have $\strongConvMFW \leq f(\y)-f(\x) - \langle \nabla f(\x), \y-\x \rangle \leq \Cf$.
\end{proof}

\begin{reptheorem}{thm:linear_convergenceMFW}
Suppose that $f$ has smoothness constant $\CfMFW$ as defined in~\eqref{eq:CfMFW}, as well as geometric strong convexity constant~$\strongConvMFW$ as defined in~(\ref{eq:muf}).
Then the error of the iterates of the FW algorithm with away-steps\footnote{%
In the algorithm, one can either use line-search or set the step-size as the feasible one that minimizes the quadratic upper bound given by the curvature $\Cf$%
, i.e. $\stepsize_k := \min\{1,\stepmax, \stepbound_k\}$ where $\stepbound_k := \frac{g_k}{2\CfMFW}$ and\vspace{-2mm} $g_k :=  \langle  -\nabla f(\x^{(k)}), \s_k - \vv_k \rangle$.
}
(Algorithm~\ref{alg:MFW})
decreases geometrically at each step that is not a drop step (i.e. when $\stepsize_k < \stepmax$), that is\vspace{-1mm}
\[
h_{k+1} \leq \left(1-\rho_f^\away\right) h_k \ ,\vspace{-1mm}
\]
where $\rho_f^\away := \frac{\strongConvMFW}{4\CfMFW}$. 
Moreover, the number of drop steps up to iteration $k$ is bounded by $k/2$. %
This yields the global linear convergence rate of $h_k \leq h_0 \exp(-\frac12 \rho_f^\away k)$.
\end{reptheorem}
\begin{proof}
The general idea of the proof is to use the definition of the geometric strong convexity constant to upper bound $h_k$, while using the definition of the curvature constant $\CfMFW$ to lower bound the decrease in primal suboptimality $h_k - h_{k+1}$ for the `good steps' of Algorithm~\ref{alg:MFW}. Then we upper bound the number of `bad steps' (the drop steps).

\emph{Upper bounding $h_k$.} In the whole proof, we assume that $\x^{(k)}$ is not already optimal, i.e. that $h_k > 0$. If $h_k = 0$, then because line-search is used, we will have $h_{k+1} \leq h_k = 0$ and so the geometric rate of decrease is trivially true in this case.\footnote{If the fixed schedule step-size is used, $h_k=0$ implies that $g_k=0$ and so $\stepsize_k = 0$ and thus $h_{k+1} = h_k$.} Let $\x^*$ be an optimum point (which is not necessarily unique). As $h_k > 0$, we have that $\innerProd{\nabla f(\x^{(k)})}{\x^*-\x^{(k)}} < 0$. We can thus apply the geometric strong convexity bound~\eqref{eq:muf} at the current iterate $\x:=\x^{(k)}$ using $\x^*$ as an optimum reference point to get (with $\overline{\stepsize} := \stepsize^\away(\x^{(k)}, \x^*)$):
\[
\begin{array}{rll}
\frac{{\overline{\stepsize}}^2}{2} \strongConvMFW \  \leq
&f(\x^*)-f(\x^{(k)}) -\left\langle \nabla f(\x^{(k)}) , \x^*-\x^{(k)} \right\rangle \\
=& -h_k - \overline{\stepsize} \left\langle  \nabla f(\x^{(k)}), \s_f(\x^{(k)}) - \vv_f(\x^{(k)})\right\rangle \\
\le &-h_k + \overline{\stepsize} \left\langle  \nabla f(\x^{(k)}), \s_k - \vv_k \right\rangle  \\
=&  -h_k + \overline{\stepsize}  g_k \ ,
\end{array}
\]
where we define $g_k :=  \left\langle  -\nabla f(\x^{(k)}), \s_k - \vv_k \right\rangle$ (note that $h_k \leq g_k$ and so $g_k$ also gives a primal suboptimality certificate). %
For the third line, we have used the definition of $\vv_f(\x)$ which implies $\left\langle  \nabla f(\x^{(k)}), \vv_f(\x^{(k)})\right\rangle \leq \left\langle  \nabla f(\x^{(k)}),\vv_k \right\rangle$. %
Therefore $h_k \le -\frac{{\overline{\stepsize}}^2}{2} \strongConvMFW + \overline{\stepsize} g_k$,
which is always upper bounded\footnote{Here we have used the trivial inequality $0 \le %
a^2-2ab+b^2$ for the choice of numbers $a:=\frac{g_k}{\strongConvMFW}$ and $b:=\overline{\stepsize}$.o} by $\frac{{g_k}^2}{2\strongConvMFW}$:\vspace{-2mm}
\begin{equation} \label{eq:hUpperBound}
h_k \leq  \frac{{g_k}^2}{2\strongConvMFW}.
\end{equation}

\emph{Lower bounding progress $h_k-h_{k+1}$.} A key aspect of the proof is to use the following observation: because of the way the direction $\dd_k$ is chosen in Algorithm~\ref{alg:MFW}, we have 
\begin{equation} \label{eq:gapDirection}
	\left\langle  -\nabla f(\x^{(k)}), \dd_k \right\rangle \geq g_k/2 , %
\end{equation}
and thus $g_k$ characterizes the quality of the direction $\dd_k$. To see this, note that $2 \left\langle  \nabla f(\x^{(k)}), \dd_k \right\rangle \leq  \left\langle \nabla f(\x^{(k)}), \dd_k^\FW\right\rangle  + \left\langle \nabla f(\x^{(k)}), \dd_k^\away\right\rangle = \left\langle \nabla f(\x^{(k)}), \dd_k^\FW+\dd_k^\away\right\rangle = -g_k$.

We first consider the case $\stepsize_\textrm{max} \geq 1$. Let $\x_\stepsize :=  \x^{(k)} + \stepsize \dd_k$ be the point obtained by moving with step-size $\stepsize$ in direction $\dd_k$, where $\dd_k$ is the one chosen by Algorithm~\ref{alg:MFW}. By using $\s := \x^{(k)} + \dd_k$ (a feasible point as $\stepsize_\textrm{max} \geq 1$)%
, $\x := \x^{(k)}$ and $\y := \x_\stepsize$ in the definition of the curvature constant~$\Cf$~\eqref{eq:Cf}, and solving for $f(\x_\stepsize)$, we get $f(\x_\stepsize) \leq f(\x^{(k)}) + \stepsize \left\langle  \nabla f(\x^{(k)}), \dd_k \right\rangle + \frac{\stepsize^2}{2} \Cf$, valid $\forall \stepsize \in [0,1]$. As~$\stepsize_k$ is obtained by line search and that $[0,1] \subseteq [0,\stepsize_\textrm{max}]$, we also have that $f(\x^{(k+1)}) = f(\x_{\stepsize_k}) \leq  f(\x_\stepsize)$ $\forall \stepsize \in [0,1]$. Combining these two inequalities, subtracting $f(\x^*)$ on both sides, and using $\Cf \leq \CfMFW$ to simplify the possibilities yields $h_{k+1} \le h_k + \stepsize \left\langle  \nabla f(\x^{(k)}), \dd_k \right\rangle + \frac{\stepsize^2}{2} \CfMFW$.

Using the crucial gap inequality~\eqref{eq:gapDirection}, we get $h_{k+1} \le h_k - \stepsize \frac{g_k}{2} + \frac{\stepsize^2}{2} \CfMFW$, and so:
\begin{equation} \label{eq:hProgress}
h_k - h_{k+1} \ge \stepsize \frac{g_k}{2} - \frac{\stepsize^2}{2} \CfMFW \quad \forall \stepsize \in [0,1].  
\end{equation}
We can minimize the bound~\eqref{eq:hProgress} on the right hand side by letting $\stepsize = \stepbound_k := \frac{g_k}{2\CfMFW}$ -- supposing that $\stepbound_k \leq 1$, we then get $h_k - h_{k+1} \geq  \frac{g_k^2}{8\CfMFW}$ (we cover the case $\stepbound_k > 1$ later).  By combining this inequality with the one from geometric strong convexity~\eqref{eq:hUpperBound}, we get 
\begin{equation} \label{eq:mainGeometric}
\frac{h_k - h_{k+1} }{h_k} \ge \frac{\strongConvMFW}{4\CfMFW} \vspace{-1mm}
\end{equation}
implying that we have a geometric rate of decrease $h_{k+1} \leq \Big(1-\frac{\strongConvMFW}{4\CfMFW}\Big) h_k$ (this is a `good step').

\emph{Boundary cases.} We now consider the case $\stepbound_k > 1$ (with $\stepsize_\textrm{max} \geq 1$ still). The condition $\stepbound_k > 1$ then translates to $g_k \geq 2 \CfMFW$, which we can use in~\eqref{eq:hProgress} with $\stepsize = 1$ to get $h_k - h_{k+1} \geq \frac{g_k}{2} - \frac{g_k}{4} = \frac{g_k}{4}$. Combining this inequality with $h_k \leq g_k$ gives the geometric decrease $h_{k+1} \leq \left(1-\frac{1}{4}\right) h_k$ (also a `good step'). $\rho_f^\away$ is obtained by considering the worst-case of the constants obtained from $\stepbound_k > 1$ and $\stepbound_k \leq 1$. (Note that always $\strongConvMFW \leq \CfMFW$ by definition, as discussed in Remark~\ref{rem:mfMFWsmallerThanCf}).

Finally, we are left with the case that $\stepsize_\textrm{max} < 1$. This is thus an away step and so $\dd_k = \dd_k^\away = \x^{(k)} - \vv_k$. Here, we use the away version $\CfA$ of the definition for $\CfMFW$: by letting $\s := \vv_k$, $\x := \x^{(k)}$ and $\y := \x_\stepsize$ in~\eqref{eq:CfA}, we also get the bound $f(\x_\stepsize) \leq f(\x^{(k)}) + \stepsize \left\langle  \nabla f(\x^{(k)}), \dd_k \right\rangle + \frac{\stepsize^2}{2} \CfMFW$, valid $\forall \stepsize \in [0,1]$ (but note here that the points $\x_\stepsize$ are not feasible for $\stepsize > \stepsize_\textrm{max}$ -- the bound considers some points outside of $\domain$). We now have two options: either $\stepsize_k = \stepsize_\textrm{max}$ (a drop step) or $\stepsize_k < \stepsize_\textrm{max}$. In the case $\stepsize_k < \stepsize_\textrm{max}$ (the line-search yields a solution in the interior of $[0,\stepsize_\textrm{max}]$), then because $f(\x_\stepsize)$ is convex in $\stepsize$, we know that $\min_{\stepsize \in [0,\stepsize_\textrm{max}]} f(\x_\stepsize) = \min_{\stepsize \geq 0} f(\x_\stepsize)$ and thus $\min_{\stepsize \in [0,\stepsize_\textrm{max}]} f(\x_\stepsize) = f(\x^{(k+1)}) \leq f(\x_\stepsize)$ $\forall \stepsize \in [0,1]$. We can then re-use the same argument above equation~\eqref{eq:hProgress} to get the inequality~\eqref{eq:hProgress}, and again considering both the case $\stepbound_k \leq 1$ (which yields inequality~\eqref{eq:mainGeometric}) and the case $\stepbound_k > 1$ (which yields $(1-\frac{1}{4})$ as the geometric rate constant), we get a `good step' with $1-\rho_f^\away$ as the worst-case geometric rate constant.

Finally, we can easily bound the number of drop steps possible up to iteration $k$ with the following argument (the drop steps are the `bad steps' for which we cannot show good progress). Let $A_k$ be the number of steps that added a vertex in the expansion (only standard FW steps can do this) and let $D_k$ be the number of drop steps. We have that $|\Coreset^{(k)}| = |\Coreset^{(0)}| + A_k - D_k$. Moreover, we have that $A_k+D_k \leq k$. We thus have $1 \leq |\Coreset^{(k)}| \leq |\Coreset^{(0)}| + k - 2D_k$, implying that $D_k \leq \frac{1}{2}( |\Coreset^{(0)}|-1+k)=\frac{k}{2}$, as stated in the theorem.
\end{proof}

\end{document}